\documentclass[11pt]{amsart}

\usepackage{amsmath,amsthm,amsfonts,amssymb,bm,graphicx }

\usepackage{graphicx}
\usepackage{enumitem}
\usepackage{xparse}
\usepackage{tikz}
\usepackage{comment}
\usepackage{bm}

\usepackage
{hyperref}
\hypersetup{colorlinks=true,citecolor=blue,linkcolor=blue,urlcolor=blue}

\makeatletter
\g@addto@macro\bfseries{\boldmath}
\makeatother

\theoremstyle{plain}
\newtheorem{theorem}{Theorem}
\newtheorem*{theorem*}{Theorem}
\newtheorem{lemma}[theorem]{Lemma}

\newtheorem{cor}[theorem]{Corollary}
\newtheorem{prop}[theorem]{Proposition}

\theoremstyle{remark}

\numberwithin{theorem}{section}

\numberwithin{equation}{section}

\def\N{\mathbb N}
\def\Z{\mathbb Z}
\def\R{\mathbb R}
\def\Q{\mathbb Q}

\def\O{\mathcal O}
\def\bt{\blacktriangle}

\begin{document}

\author{Magdal\'ena Tinkov\'a}

\title{Trace and norm of indecomposable integers in cubic orders}

\address{Charles University, Faculty of Mathematics and Physics, Department of Algebra,
Sokolovsk\'{a} 83, 18600 Praha 8, Czech Republic}

\email{tinkova.magdalena@gmail.com}


\thanks{The author was supported by Czech Science Foundation GA\v{C}R, grant 21-00420M, by projects PRIMUS/20/SCI/002, UNCE/SCI/022, GA UK 1298218 from Charles University, and by SVV-2020-260589.}


\begin{abstract}

We study the structure of the codifferent and of additively
indecomposable integers in families of totally real cubic fields.
We prove that for cubic orders in these fields, the minimal trace of
indecomposable integers multiplied by totally positive elements of the
codifferent can be arbitrarily large. This is very surprising, as in
the so far studied examples of quadratic and simplest cubic fields, this
minimum is 1 and 2. We further give sharp upper bounds on the norms of
indecomposable integers in our families.
  
\end{abstract}

\setcounter{tocdepth}{2}  \maketitle 

\section{Introduction}

The interest in indecomposable integers in totally real number fields comes from the study of universal quadratic forms. In 1770, Lagrange proved his famous four-square theorem, which says that every positive integer can be expressed as the sum of four squares, i.e., for every $n\in\N$ there exist $x,y,z,w\in\Z$ such that $n=x^2+y^2+z^2+w^2$. A quadratic form satisfying this condition is called \textit{universal}. This property was then studied for other quadratic forms and consequently put into a more general setting as follows. Instead of the set of rational integers, we consider an arbitrary totally real number field $K$, and the set of positive integers is replaced by the set of totally positive elements of $K$. By this, we mean the elements whose images are positive for every embedding of $K$ into $\R$.

Even more generally, we can restrict to totally positive elements of the ring of algebraic integers $\O_K$, or consider any order in $K$. When we study universal quadratic forms over $K$ and especially the minimal number of variables that these quadratic forms must have, it is useful to know some (or the whole structure of) \textit{additively indecomposable integers}. By this, we mean those elements that are totally positive and cannot be decomposed as the sum of totally positive elements. Note that in contrast with general number fields, the structure of indecomposable integers in $\Z$ is trivial since there is only one indecomposable integer, namely $1$. 

For the study of universal quadratic forms, indecomposable integers were used, for example, by Blomer, Kala, Yatsyna, and other authors \cite{BK,CLSTZ,Ka,KT,KY,KTZ,Ya}. Nevertheless, under the name \textit{extremal elements}, they already appeared in Siegel's paper \cite{Si}, where he proved that the sum of any number of squares can be universal only over the fields $\Q$ and  $\Q(\sqrt{5})$.

Despite their usefulness, we do not know much about them with the exception of the case of real quadratic fields. Perron \cite{Pe}, and Dress and Scharlau \cite{DS} found all indecomposable integers in fields $\Q(\sqrt{D})$, where these elements can be obtained from the continued fraction of $\sqrt{D}$ and $\frac{\sqrt{D}-1}{2}$, respectively. Moreover, Hejda and Kala investigated them from the point of view of their additive structure \cite{HK}. 

\v{C}ech, Krásenský, Lachman, Svoboda, Zemková and the present author \cite{CLSTZ,KTZ} studied the question whether indecomposable integers from quadratic subfields of a biquadratic field $\Q(\sqrt{p},\sqrt{q})$ remain indecomposable in $\Q(\sqrt{p},\sqrt{q})$. However, their results were only partial.

A much needed breakthrough in the study of indecomposable integers in fields of higher degrees was achieved by Kala and the present author \cite{KT}, where they fully characterized indecomposable integers in Shanks' family of the simplest cubic fields \cite{Co,Sh} and obtained important consequences for universal quadratic forms. These fields are generated by a root of the polynomial $x^3-ax^2-(a+3)x-1$ with  $a\geq -1$ and, due to their simple structure and other useful properties (they are, for example, cyclic), were studied by many researchers, see, for example, \cite{Ba,By,Fo,Kis,LP,Let,Lo,Wa}. Moreover, $\O_K=\Z[\rho]$ in infinitely many cases of $a$. 
This paper closely follows and extends our results in \cite{KT}.

Let $\O\subseteq\O_K$ be an order in a totally real number field $K$, and let $\O^{\vee,+}$ be the set of totally positive elements $\delta$ of $K$ such that the trace $\textup{Tr}(\alpha\delta)$ is a rational integer for every $\alpha\in\O$ (for more details, see Section \ref{Sec:Prelim}).   
 In \cite{KT}, we studied the possible values of $\min_{\delta\in\O_K^{\vee,+}}\textup{Tr}(\alpha\delta)$ where $\alpha$ is an indecomposable integer in $\O_K$ as it is crucial for universal quadratic forms. We showed that it is equal to $1$ for all indecomposable integers in quadratic fields and stated several results about the behavior of this value for several families of cubic fields. In particular, it is equal to $1$ except for one exceptional indecomposable integer (for which it is $2$) in the case of the simplest cubic fields. 

Moreover, we conjectured that it is $2$ for all non-unit indecomposable integers in the family of Ennola's cubic fields; we will prove it in this paper in Section \ref{Sec:Ennola}. In \cite[Subsection 8.2]{KT}, we also showed an example of a cubic field and an indecomposable integer with this value equal to $3$. From this, inspired by the results on quadratic fields, it might seem that this value is bounded for cubic fields, for example, by $3$. Nevertheless, at least for the case of cubic orders, we will show that it is not true by proving the following surprising theorem:

\begin{theorem} \label{thm:mintraces}
For every $n\in\N$ there exists a cubic unit $\rho$ and an indecomposable integer $\alpha\in\Z[\rho]$ such that
\[
\min_{\delta\in\Z[\rho]^{\vee,+}}\textup{Tr}(\alpha\delta)>n.
\]
\end{theorem}
Note that the minimal traces of the codifferent itself were in greater detail studied for the case of the simplest cubic fields \cite{KH}, and a few examples can also be found in \cite{Cohen}. 

Besides that, one particular property of indecomposable integers was studied repeatedly. In \cite{Bru}, Brunotte proved that the norm $N(\alpha)$ of indecomposable integers $\alpha$ is bounded in each totally real number field. However, this bound is often too large to be useful for applications or the determination of indecomposable integers. The first further results on this point can be already found in the work of Dress and Scharlau \cite{DS}, who derived a much sharper upper bound for quadratic fields; in some cases, this bound cannot be lowered. Their result was later refined by Jang and Kim \cite{JK}. Moreover, they stated a conjecture improving their bound, which was later disproved \cite{Ka2,TV}. Nevertheless, Voutier and the present author \cite{TV} came up with an even shaper bound.

Inspired by this interest, we will study this bound for several cases of cubic fields. For the case of the simplest cubic fields, we will prove the following statement:     

\begin{theorem}\label{thm:Mainnorm}
Let $K=\Q(\rho)$ be a field generated by a root of the polynomial \[x^3-ax^2-(a+3)x-1\] where $a\geq -1$. Moreover, assume that $\O_K=Z[\rho]$ and $a=3A+a_0$ where $A\in\N_0$ and $a_0\in\{0,1,2\}$. If $\alpha$ is indecomposable in $\O_K$, then 
	\[
	N(\alpha)\leq\left\{
	\begin{array}{ll}
	a^2+3a+9&\text{ if } a\leq 3,\\
	3A^4+14A^3+28A^2+27A+11&\text{ if } a_0=2,\\
	3A^4+10A^3+16A^2+13A+5&\text{ if } a_0=1,\\
	3A^4+6A^3+9A^2+6A+3&\text{ if } a_0=0.
	\end{array}
	\right.,
	\] 
\end{theorem}

Note that this bound is always the norm of some indecomposable integer, so it is sharp. Moreover, we will find similar upper bounds for several other cubic orders as well as show in more detail some results, which were announced without proofs in \cite{KT}.

This paper is organized as follows. Section \ref{Sec:Prelim} contains some basic facts  about our problem and tools which we will use. In Section \ref{Sec:simpl}, we will look more closely at the case of the simplest cubic fields and prove Theorem \ref{thm:Mainnorm}. Ennola's cubic fields are in the center of interest in Section \ref{Sec:Ennola}; we will show there several results, including the derivation of an upper bound on the norm of indecomposable integers.

In Section \ref{Sec:excunits}, we will study a family of cubic orders, which was introduced by Thomas \cite{Th}. First of all, we will find the structure of their indecomposable integers (see Proposition \ref{prop:indeexcunits}). Then, benefiting from the properties of this family, we will prove Theorem \ref{thm:mintraces}, and as for the other cases, we will find an upper bound on the norm of indecomposable integers.

\section{Preliminaries} \label{Sec:Prelim}

Let $K$ be a totally real number field with the ring of algebraic integers $\O_K$. We say that an element $\alpha\in K$ is \textit{totally positive} if $\sigma(\alpha)>0$ for every embedding $\sigma$ of $K$ into $\R$; we denote it by $\alpha\succ 0$. By trace $\text{Tr}(\alpha)$ and norm $N(\alpha)$ of $\alpha\in K$, we mean 
\[
\text{Tr}(\alpha)=\sum_{\sigma}\sigma(\alpha)\qquad \text{ and }\qquad N(\alpha)=\prod_{\sigma}\sigma(\alpha)
\]
where $\sigma$ runs over all the embeddings $\sigma$. In this paper, we will mostly work with cubic fields $K=\Q(\rho)$. The elements $\rho'$ and $\rho''$ are  the conjugates of $\rho$, i.e., its other images in the embedding $\sigma$. 
If $\alpha=a_1+a_2\rho+a_3\rho^2\in K$, we will denote by $\alpha'$ and $\alpha''$ the corresponding images $\alpha'=a_1+a_2\rho'+a_3\rho'^2$ and $\alpha''=a_1+a_2\rho''+a_3\rho''^2$.   

Let $\O\subseteq \O_K$ be an order in $K$. We denote the subset of totally positive elements of $\O$ by $\O^{+}$. We say that $\alpha\in\O^{+}$ is indecomposable in $\O$ if $\alpha$ cannot be written as $\alpha=\beta+\gamma$ for any $\beta,\gamma\in\O^{+}$.

The \textit{codifferent} is
\[
\O^{\vee}=\{\delta\in K; \text{Tr}(\alpha\delta)\in\Z\text{ for all }\alpha\in\O\}.
\]
In the case that $\O=\Z[\eta]$ for some $\eta\in K$, where the minimal polynomial of $\eta$ is $f$, we have \cite[Proposition 4.17]{Na} 
\[
\O^{\vee}=\frac{1}{f'(\eta)}\Z[\eta].
\]
The symbol $\O^{\vee,+}$ stands for the subset of totally positive elements of $\O^{\vee}$. It can be easily seen that if for $\alpha\in\O^{+}$ there exists $\delta \in\O^{\vee,+}$ such that $\text{Tr}(\alpha\delta)=1$, then $\alpha$ is indecomposable in $\O$.

In \cite[Section 4]{KT}, we developed a method for the determination of the structure of indecomposable integers. We will recall here only some basic facts relevant to the case of totally real cubic fields. Let $\varepsilon_1$ and $\varepsilon_2$ two totally positive units in $K$ which are \textit{proper} in the sense of Thomas and Vasquez (see, \cite[Section 1]{ThV}, and in particular \cite[Proposition 1, Corollary 2]{ThV}; or \cite[Section 3]{Ok}).  
Moreover, let $\{\omega_1, \omega_2,\omega_3\}$ be an integral basis of $K$, and let $\gamma=x_1\omega_1+x_2\omega_2+x_3\omega_3\in K$.
We will consider the embedding of the following form:
\begin{align*}
\tau :\;&K\rightarrow\R^3,\\
&\gamma \mapsto (x_1,x_2,x_3).
\end{align*} 

Then for every indecomposable integer $\alpha$ in $\O$ there exists a totally positive unit $\varepsilon$ such that  
\[
\tau(\alpha\varepsilon)\in\{t_1\tau(1)+t_2\tau(\varepsilon_1)+t_3\tau(\varepsilon_2);\; 0\leq t_1,t_2,t_3\leq 1\},
\]
or 
\[
\tau(\alpha\varepsilon)\in\{t_1\tau(1)+t_2\tau(\varepsilon_1)+t_3\tau(\varepsilon_1\varepsilon_2^{-1});\; 0\leq t_1,t_2,t_3\leq 1\}.
\]
Thus, to determine the structure of indecomposable integers in $\O$, it suffices to study a finite number of elements of $\O$ whose images under $\tau$ lie in one of these two parallelepipeds. However, not all of these elements are indecomposable. Thus we have to conduct some further investigation using the codifferent to resolve that. Note that due to this construction, we can immediately see that all the elements we get this way are totally positive. For more details about this method, see \cite{KT}.  


\section{The simplest cubic fields} \label{Sec:simpl}

The goal of this section is to find a sharp bound on the norm of indecomposable integers in the simplest cubic fields $K=\Q(\rho)$ \cite{Co, Sh}. They are generated by a root $\rho$ of the polynomial $x^3-ax^2-(a+3)x-1$ where $a\geq -1$. For these fields, $\O_K=\Z[\rho]$ in infinitely many cases of $a$, e.g., when the square root of their discriminant $a^2+3a+9$ is square-free. They are Galois and possess units of all signatures. Moreover, every totally positive unit is a square in $K$. 

The determination of an upper bound on the norm of indecomposable integers in $K$ is simplified by the fact that we know the precise structure of indecomposable integers in $\Q(\rho)$. It was found in \cite{KT}, where we proved the following theorem:

\begin{theorem}[{\cite[Theorem 1.2]{KT}}]
Let $K$ be the simplest cubic field with
$a\in\Z_{\geq -1}$ such that $\O_K=\Z[\rho]$. 
The elements $1$, $1+\rho+\rho^2$, and $-v-w\rho+(v+1)\rho^2$ where $0\leq v\leq a$ and $v(a+2)+1\leq w\leq (v+1)(a+1)$ are, up to multiplication by totally positive units, all the indecomposable integers in $\Q(\rho)$.
\end{theorem}

As was shown in \cite[Subsection 5.2]{KT}, the indecomposable integer $1+\rho+\rho^2$ (with norm $a^2+3a+9$) differs from the rest of the indecomposable integers in $K$. However, as we will see below, except for some small values of $a$, it does not play an important role in the study of the upper bound on the norm of indecomposable integers.

Thus, let us now focus on the remaining indecomposable integers. Excluding totally positive units with the norm $1$, we are left with the set 
\[
\bt(a)=\{-v-w\rho+(v+1)\rho^2, 0\leq v\leq a \text{ and } v(a+2)+1\leq w\leq (v+1)(a+1)\}.
\]
Let $-v-w\rho+(v+1)\rho^2=-v-(v(a+2)+1+W)\rho+(v+1)\rho^2=\alpha(v,W)$ where $0\leq W\leq a-v$. Let us define the transformation $T_1:\bt(a)\rightarrow\bt(a)$ as
\[
T_1(\alpha(v,W))=(\alpha(v,W))'(-a-1-(a^2+3a+3)\rho+(a+2)\rho^2)=\alpha(W,a-v-W).
\]
Note that $-a-1-(a^2+3a+3)\rho+(a+2)\rho^2$ is a totally positive unit in $K$. Similarly, $T_2:\bt(a)\rightarrow\bt(a)$ is defined as
\[
T_2(\alpha(v,W))=(\alpha(v,W))''\rho^2=\alpha(a-v-W,v).
\] 
It can be easily shown that the images $T_1(\alpha(v,W))$ and $T_2(\alpha(v,W))$ always lie in $\bt(a)$, and $\alpha(v,W)=T_1(\alpha(v,W))=T_2(\alpha(v,W))$ only if $3|a$ and $v=W=\frac{a}{3}$. This particular result tells us that besides the indecomposable integer $\alpha(v,W)$, the set $\bt(a)$ also contains unit multiples of its conjugates, which coincide with $\alpha(v,W)$ only in the above-mentioned exceptional case. Moreover, these elements have the same norm as $\alpha(v,W)$, so we can restrict to some subset of $\bt(a)$ to find an upper bound on norm of indecomposable integers in $\bt(a)$.   

To specify this restriction, let $a=3A+a_0$ where $a_0\in\{0,1,2\}$. Put
{\small
\[
\bt_0(a)=\left\{
\begin{array}{ll}
\left\{\alpha(v,W);0\leq v\leq A\text{ and } v\leq W\leq a-2v-1\right \}&\text{ if } a_0\in\{1,2\},\\
\{\alpha(v,W);0\leq v\leq A-1\text{ and } v\leq W\leq a-2v-1\}\cup\{\alpha(A,A)\}&\text{ if } a_0=0.
\end{array}
\right.
\]}
It can be easily shown that if $\alpha\in\bt_0(a)$ and $\alpha\neq \alpha\big(\frac{a}{3},\frac{a}{3}\big)$, then $T_1(\alpha),T_2(\alpha)\notin \bt_0(a)$. Moreover, for every $\alpha\in\bt(a)$ the set $\bt_0(a)$ contains at least one of the elements $\alpha,T_1(\alpha)$ and $T_2(\alpha)$.   

It can be easily derived that
\begin{multline*}
N(\alpha(v,W))=3 + 6 A + 2 a_0 + 9 A v + 9 A^2 v + 3 a_0 v + 6 A a_0 v + a_0^2 v - 
3 v^2 - 6 A v^2 \\- 2 a_0 v^2 + v^3 + 9 A W + 3 a_0 W - 3 v W + 
3 A v W + 9 A^2 v W + a_0 v W + 6 A a_0 v W + a_0^2 v W \\- 3 A v^2 W - 
a_0 v^2 W - 3 W^2 + 3 A W^2 + a_0 W^2 - 3 v W^2 - 3 A v W^2 - 
a_0 v W^2 - W^3
\end{multline*}

In the following lemma, we will compare the norms of $\alpha(v,W)$ for some particular choices of $v$ and $W$. Note that after each step, we can exclude some subset of elements of $\bt_0(a)$ from consideration; their norm is too small to be the maximumal norm of $\bt(a)$. The first part of this lemma (with $a\geq 3$) originally appeared in \cite[Lemma 6.4]{KT} with a sketch of the proof. For convenience and illustration for other parts of the lemma, we show here its extended version. 

\begin{lemma} \label{lem:comsimplest}
	Let $a\geq 5$. Then 
	\begin{enumerate}
		\item $N(\alpha(v,W))<N(\alpha(v+1,W))$ for all $0\leq v \leq A-1$ and $v+1\leq W\leq 3A+a_0-2v-3$,
		\item $N(\alpha(v,v))<N(\alpha(v+1,v+1))$ for all $0\leq v \leq A-1$,
		\item $N(\alpha(v,3A+a_0-2v-1))<N(\alpha(v+1,3A+a_0-2(v+1)-1))$ for all $0\leq v \leq A-2$,
		\item $N(\alpha(v,3A+a_0-2v-2))<N(\alpha(v+1,3A+a_0-2(v+1)-2))$ for all $0\leq v \leq A-2$.   
	\end{enumerate}
\end{lemma}

\begin{proof}
	To prove (1), we can easily compute that
	\begin{multline*}
	N(\alpha(v+1,W))-N(\alpha(v,W))=-2 + 3 A + 9 A^2 + a_0 + 6 A a_0 + a_0^2 - 3 v - 12 A v - 4 a_0 v + 
	3 v^2 \\- 3 W + 9 A^2 W + 6 A a_0 W + a_0^2 W - 6 A v W - 2 a_0 v W - 
	3 W^2 - 3 A W^2 - a_0 W^2.
	\end{multline*}
	This expression as a quadratic polynomial in $W$, having a negative coefficient before $W^2$, can be positive only on some bounded interval. Thus if we prove that it is positive for $W=v+1$ and $W=3A+a_0-2v-3$, we show that this polynomial takes positive values for all the coefficients $W$ between $v+1$ and $3A+a_0-2v-3$.
	
	For $W=v+1$, we get
	\[
	-8 + 18 A^2 + 12 A a_0 + 2 a_0^2 - 12 v - 24 A v + 9 A^2 v - 8 a_0 v + 
	6 A a_0 v + a_0^2 v - 9 A v^2 - 3 a_0 v^2
	\]
	This is also a quadratic polynomial in $v$ with a negative leading coefficient, so it suffices to check our polynomial at $v=0$ and $v=A-1$. Putting $v=0$, we obtain
	\[
	-8 + 18 A^2 + 12 A a_0 + 2 a_0^2,
	\]
	which is positive for $a\geq 5$. The choice $v=A-1$ leads to
	\[
	4 + 3 A + 3 A^2 + 5 a_0 + 4 A a_0 + 3 A^2 a_0 + a_0^2 + A a_0^2>0.
	\]
	Thus the initial polynomial in $W$ is positive for $W=v+1$. Similarly, for $W=3A+a_0-2v-3$, we get the polynomial
	\[
	-20 + 21 A + 9 A^2 + 7 a_0 + 6 A a_0 + a_0^2 - 33 v + 6 A v + 2 a_0 v - 
	9 v^2,
	\]
	which have the same properties as the previous quadratic polynomial in $v$. For $v=0$, we obtain
	\[
	-20 + 21 A + 9 A^2 + 7 a_0 + 6 A a_0 + a_0^2>0
	\]
	for $a\geq 5$. Likewise, $v=A-1$ implies the expression of the form
	\[
	4 + 6 A^2 + 5 a_0 + 8 A a_0 + a_0^2
	\]
	which is also positive for $a\geq 5$. Thus we have completed the proof of (1). The proof of the other parts is analogous, and we can use the same methods as here.	
\end{proof}

The previous lemma significantly limits the candidates on the element of $\bt_0(a)$ with the largest norm. We are left with $\alpha(A-1, 3A+a_0-2(A-1)-2)$, $\alpha(A-1, 3A+a_0-2(A-1)-1)$ (except for $3|a$ where one of these elements can be excluded by applying (1)) and with the elements having $v=A$. Considering this, we get that the largest norm is obtained by one of the following elements:
\begin{itemize}
	\item $\alpha(A-1,A+1)$, $\alpha(A,A)$ or $1+\rho+\rho^2$ for $a_0=0$,
	\item $\alpha(A-1,A+1)$, $\alpha(A-1,A+2)$, $\alpha(A,A)$ or $1+\rho+\rho^2$ for $a_0=1$,
	\item $\alpha(A-1,A+2)$, $\alpha(A-1,A+3)$, $\alpha(A,A)$, $\alpha(A,A+1)$ or $1+\rho+\rho^2$ for $a_0=2$. 
\end{itemize}

Now we have everything we need to prove Theorem \ref{thm:Mainnorm} stated in the introduction. 

\begin{proof}[Proof of Theorem \ref{thm:Mainnorm}]
	
	One can immediately show that for $-1\leq a\leq 3$, the element $1+\rho+\rho^2$ has the largest norm among the indecomposable integers in $\Q(\rho)$. This is not true for $a=4$. Here the largest norm is $47$, which can be expressed by the given formula.
	
	Comparing the norms of the elements listed above, we can show that the largest norm of $\bt_0(a)$ is attained by $\alpha(A,A)$ for $a_0\in\{0,1\}$, and by $\alpha(A,A+1)$ for $a_0=2$. If we compare this maximum with the norm of $1+\rho+\rho^2$, which is $a^2+3a+9$, it follows that it is larger for $a\geq 4$.  
\end{proof}

\section{Ennola's cubic fields} \label{Sec:Ennola}

Now we turn our attention to the so-called Ennola's cubic fields, which were introduced by Ennola \cite{En2}. They are generated by a root $\rho$ of the polynomial $x^3+(a-1)x^2-ax-1$ where $a\geq 3$. It can be easily shown that this polynomial has roots $\rho, \rho'$ and $\rho''$ satisfying
\begin{align*}
1+\frac{1}{a+3}&<\rho<1+\frac{1}{a+2},\\
-\frac{1}{a}&<\rho'<-\frac{1}{a+1},\\
-a+\frac{1}{a^2+a}&<\rho''<-a+\frac{1}{a^2}.
\end{align*}
Ennola's cubic fields are interesting from the point of view of their system of fundamental units, which is formed by the pair $\rho$ and $\rho-1$.

The indecomposable integers of $\Z[\rho]$ were also studied in \cite{KT}, where we proved the following proposition.

\begin{prop}[{\cite[Proposition 8.1]{KT}}] \label{prop:ennolainde}
Let $\rho$ be a root of the polynomial $x^3+(a-1)x^2-ax-1$ where $a\geq 3$. Then $1$ and $1+w\rho+\rho^2$ where $1\leq w\leq a-1$ are, up to multiplication by totally positive units, all the indecomposable integers in $\Z[\rho]$. 
\end{prop}

Now we will focus on $\min_{\delta\in\Z[\rho]^{\vee,+}}\text{Tr}(\alpha\delta)$ for indecomposable integers $\alpha$. In \cite[Subsection 8.1]{KT}, we proposed that it should be equal to $2$ for all non-unit indecomposable integers. We prove it in the following proposition. Note that it differs from the situation in the simplest cubic fields; up to multiplication by totally positive units, they have $\min_{\delta\in\O_K^{\vee,+}}\text{Tr}(\alpha\delta)=1$ for all but one indecomposable integer \cite[Section 5]{KT}.  

\begin{prop}
Let $\alpha$ be a non-unit indecomposable integer in $\Z[\rho]$. Then \[\min_{\delta\in\Z[\rho]^{\vee,+}}\textup{Tr}(\alpha\delta)=2.\] 
\end{prop}

\begin{proof}
Since multiplication by totally positive units does not influence the resulting value of our minimum, we can restrict to the indecomposable integers listed in Proposition \ref{prop:ennolainde}. 
As was proved in \cite[Subsection 8.1]{KT}, for them, there exists $\delta\in\Z[\rho]^{\vee,+}$ such that \[\text{Tr}((1+w\rho+\rho^2)\delta)=2\] for all $1\leq w\leq a-1$. Thus, it suffices to prove that there is no totally positive element $\frac{\delta_1}{f'(\rho)}\in\Z[\rho]^{\vee,+}$ with $\text{Tr}\big(\frac{\delta_1}{f'(\rho)}(1+w\rho+\rho^2)\big)=1$. Note that this element $\delta_1$ can be different for each of the considered indecomposable integers. 

On the contrary, suppose that such an element exists.
Put $\delta_1=c+d\rho+e\rho^2\in\Z[\rho]$. It can be easily computed that
\[
\text{Tr}\bigg(\frac{\delta_1}{f'(\delta)}(1+w\rho+\rho^2)\bigg)=c + d (1 - a + w) + e (2 + a^2 + w - a (1 + w)).
\] 
That is equal to $1$ only if
\[
\delta_1=1+k(1-a+w)+l(2+a^2+w-a(1+w))-k\rho-l\rho^2
\]
for some $k,l\in\Z$. We can also deduce that
\[
\text{Tr}\bigg(\frac{\delta_1}{f'(\rho)}\bigg)=-l.
\] 
The element $\frac{\delta_1}{f'(\rho)}$ is supposed to be totally positive, and thus $l<0$.

It can be easily shown that $f'(\rho),(f'(\rho))''>0$, and $(f'(\rho))'<0$. Note that the element $\delta_1$ must have the same signature as $f'(\rho)$. Using some basic estimates on $\rho,\rho'$ and $\rho''$, we can directly conclude that this is possible only if $k<0$. 

We will proceed with the further development of this requirement. We have
\[
0<\delta_1<1+k\left(w-a-\frac{1}{a+2}\right)+l\left(a^2-a+1-aw+w-\frac{2}{a+2}-\frac{1}{(a+2)^2}\right)=\psi_1(k,l).
\] 
We see that the coefficient after $k$ is negative (as we have $w\leq a-1$), and the coefficient after $l$ is positive. For negative $k,l$, it means that $k$ has to be sufficiently small in comparison to $l$ to get $\psi_1(k,l)>0$.

Moreover, we can compute that
\[
(a+2)^2\psi_1(al-l+1,l)=-a^3-3a^2-a+2+(a^2+4a+4)w+l(a+1)=\varphi(w).
\]
This is a linear polynomial in $w$ with a positive leading coefficient. Therefore, if we show that $\varphi(a-1)<0$, then $\varphi(w)<0$ for every $w\leq a-1$. 
We can immediately see that
\[
\varphi(a-1)=-a-2+l(a+1)<0
\] 
for $l<0$.
Thus $\psi_1(al-l+1,l)<0$ for every $1\leq w\leq a-1$. To get $\delta_1>0$, we thus must have $k<al-l+1$.

Let us now turn to $\delta_1''$, which has to be positive as well. We can conclude that
\[
0<\delta_1''<1+k\left(1-\frac{1}{a^2}+w\right)+l\left(2-a+w-aw+\frac{2a}{a^2+a}-\frac{1}{(a^2+a)^2}\right)=\psi_2(k,l).
\] 
The coefficient after $k$ is positive, thus, using the previous part, we can make the following estimates:
\[
(a^2+a)^2\psi_2(k,l)\leq (a^2+a)^2\psi_2(al-l,l)=a^4+2a^3+a^2+l(a^4+3a^3+2a^2+a). 
\] 
However, the expression on the right side is negative for any $a\geq 3$ and $l<0$.

Thus $\frac{\delta_1}{f'(\rho)}$ cannot be totally positive for any choice of parameters $k,l$, which completes the proof.
\end{proof}

Let $1+w\rho+\rho^2=\alpha(w)$.
One can easily prove that the norm of $\alpha(w)$ is equal to 
\[
N(\alpha(w))=2a^2+2a+1+(a^2-3a-2)w-(2a-1)w^2+w^3.
\]
In comparison to the simplest cubic fields, the structure of indecomposable integers in Ennola's cubic fields (or in the considered orders) is simple, and we can find an upper bound on their norms more easily. As before, the following bound is, in fact, the norm of some specific indecomposable integer in $\Z[\rho]$, so it cannot be lowered.

\begin{prop}
Let $a=3A+a_0$ where $A\in\N_{0}$ and $a_0\in\{0,1,2\}$. If $\alpha$ is indecomposable in $\Z[\rho]$, then
\[
N(\alpha)\leq\left\{
                \begin{array}{ll}
                  3a^2-3a+1&\text{ if } a\leq 4,\\
                  4A^3+18A^2+26A+13&\text{ if } a_0=2,\\
                  4A^3+14A^2+16A+7&\text{ if } a_0=1,\\
                  4A^3+10A^2+8A+3 &\text{ if } a_0=0.
                  
                \end{array}
              \right.,
\]
\end{prop}

\begin{proof}
Let $\alpha(w)=1+w\rho+\rho^2$. We have
\[
N(\alpha(w+1))-N(\alpha(w))=-5 a + a^2 + 5 w - 4 a w + 3 w^2.
\]
This polynomial in $w$ has two roots, namely $w=\frac{a-5}{3}$ and $w=a$. Recall that $1\leq w\leq a-1$. It follows that up to $\frac{a-5}{3}$, this polynomial is positive, i.e., $N(\alpha(w+1))>N(\alpha(w))$. On the other hand, between $\frac{a-5}{3}$ and $a$, it is negative, which implies $N(\alpha(w+1))<N(\alpha(w))$. Thus the maximal norm of indecomposable integers is attained for some integer $w$ close to $\frac{a-5}{3}$. 

If $\frac{a-5}{3}<1$, i.e., if $a<8$, the maximum is at the point $w=1$, which corresponds to the formulas given in the statement. If $a\geq 8$, we have to distinguish our cases according to the value $a$ modulo $3$. We see that $\frac{a-5}{3}=\frac{3A+a_0-5}{3}=A-\frac{5-a_0}{3}$. Hence the maximal value of the norm is attained in $w=A-1$, which is true for all of three cases of $a_0$ (note that for $a_0=2$, $N(\alpha(A-1))=N(\alpha(A))$). Thus the formulas in the statement of this proposition are norms of the indecomposable integer $1+w\rho+\rho^2$ with $w=A-1$.
\end{proof}

\section{Orders with exceptional units} \label{Sec:excunits}

In this section, we will focus on orders $\Z[\rho]$ where $\rho$ is a root of a cubic polynomial of the form
\[
f(x)=x^3-(a+b)x^2+abx-1
\]
where $2\leq a\leq b-2$.  This family is again interesting due to the system of fundamental units of $\Z[\rho]$, which is formed by $\rho$ and $\rho-a$ \cite[Theorem (3.9)]{Th}. Moreover, the roots $\rho,\rho'$ and $\rho''$ of $f$ satisfy $b<\rho<b+1$, $a-1<\rho'<a$ and $\frac{1}{ab}<\rho''<\frac{1}{ab-1}$, which one can easily check. 

In \cite[Proposition 8.3]{KT}, we announced without proof a proposition about the structure of indecomposable integers for one subfamily of this form (i.e., for $b=a+2$). In the following, we provide more general results, from which the statement in \cite{KT} follows. 

For the case of the simplest cubic fields and Ennola's cubic fields, we know that \[\min_{\delta\in\Z[\rho]^{\vee,+}}\textup{Tr}(\alpha\delta)\leq 2\] for every indecomposable integer $\alpha$ in the corresponding order. Nevertheless, this is not true in general for this case. In particular, we will show that $\min_{\delta\in\Z[\rho]^{\vee,+}}\textup{Tr}(\alpha\delta)$ can attain arbitrarily large values. Note that among $4\leq b\leq 100$ and $2\leq a\leq b-2$ ($4753$ pairs), we have $4264$ cases with $\O_K=\Z[\rho]$, thus these results may even suggest us something about the behavior of indecomposable integers in $\Q(\rho)$. 
 
First of all, we will find the structure of indecomposable integers in $\Z[\rho]$. To do that, we will use the method which was described at the end of Section \ref{Sec:Prelim} and involves two parallelepipeds generated by triples of totally positive units. Note that $\rho$ is a totally positive unit, and $\rho-a$ is not, thus the set of such units is generated by the pair $\rho$ and $(\rho-a)^2$. 

The first parallelepiped which we will consider originates from the triple $1$, $\rho$ and $\rho(\rho-a)^2=1+(a^2-ab)\rho+(b-a)\rho^2$. Note that Thomas and Vasquez showed that the pair $\rho$ and $\rho(\rho-a)^2$ is \textit{proper} \cite[Section 3]{ThV}, thus we can use them in the following consideration.
The first parallelepiped produces candidates on indecomposable integers of the form $1-au\rho+u\rho^2$ where $1\leq u\leq b-a-1$. However, these elements lie on the border of this parallelepiped, and their unit multiples can be found in the second parallelepiped. For this reason, we can exclude them from the following consideration.

Now we turn our attention to the second parallelepiped generated by units $1$, $\rho$ and $(\rho-a)^{-2}=a-b+(-ba^2+b^2a+1)\rho+(a^2-ab)\rho^2$. It contains elements of the form 
\[
-v+(bw+1)\rho-w\rho^2 \text{ where } 0\leq v\leq b-a-1 \text{ and } va\leq w\leq(v+1)a-1.
\]
Recall that these elements are totally positive since they can be expressed as a linear combination of totally positive units with non-negative coefficients. 
Note that for the choice $v=w=0$, we get the totally positive unit $\rho$, which we, for simplicity, do not exclude from the following consideration. In what follows, the symbol $\Psi$ stands for the set of indecomposable integers originating from the second parallelepiped, which in fact cover all the indecomposable integers in $\Z[\rho]$.

\subsection{Indecomposability} \label{subsec:indecom}
As in \cite{KT}, we will use totally positive elements of the codifferent to show that the elements belonging to $\Psi$ are indecomposable in $\Z[\rho]$. However, in comparison to the simplest cubic fields and Ennola's cubic fields, the proof will be much more complicated. 

We will consider the totally positive element of the codifferent of the form 
\[
\delta_{v}=\frac{a^2-a-(2a-1)\rho+\rho^2}{f'(\rho)}.
\]
Together with one another such element, it will play an important role in the next subsection. Let $v_1+v_2\rho+v_3\rho^2\in\Z[\rho]$. It can be easily computed that
\[
\text{Tr}(\delta_v(v_1+v_2\rho+v_3\rho^2))=v_1+(b-a+1)v_2+(b-a+1)bv_3.
\] 
Thus, for $-v+(bw+1)\rho-w\rho^2$, we get the trace $b-a+1-v$. Its value can vary between $2$ and $b-a+1$, and it is not the smallest possible for all these elements, as we will see in Subsection \ref{subsec:mintraces}. However, it is sufficient for the proof of the indecomposability of our elements.

If the element $-v+(bw+1)\rho-w\rho^2$ were decomposable, it could be representable as the sum of two totally positive elements of the form 
\[
\beta(t,v_2,v_3)=t-(b-a+1)v_2-(b-a+1)bv_3+v_2\rho+v_3\rho^2
\]
where $v_2,v_3\in\Z$ and $1\leq t\leq b-a-v$. In particular, $t\leq b-a$ and indicates the trace $\text{Tr}(\delta_v\beta(t,v_2,v_3))$, which must be lower than $b-a+1-v$. In the following lemma, we will show that in most of the cases, the elements $\beta(t,v_2,v_3)$ are not totally positive.

\begin{lemma} \label{lem:excnegat}
The elements $\beta(t,v_2,v_3)$ are not totally positive if 
\begin{enumerate}
\item $v_2=0$ and $v_3\neq 0$,
\item $v_2<0$ and $v_3<0$,
\item $v_2>0$ and $v_3\geq 0$,
\item $v_2<0$ and $v_3\geq 0$ if moreover $b\leq 2a-3$.
\end{enumerate}
\end{lemma}

\begin{proof}
For $v_2=0$, we have $\beta(t,0,v_3)=t-(b-a+1)bv_3+v_3\rho^2$. If $v_3>0$, we get
\[
(\beta(t,0,v_3))''<b-a-(b-a+1)bv_3+v_3<0.
\]
On the other hand, for $v_3<0$, we can see that
\[
\beta(t,0,v_3)<b-a-(b-a+1)bv_3+v_3b^2=b-a+(a-1)bv_3<0.
\]
Hence $\beta(t,0,v_3)$ is not totally positive except for $v_3=0$.

In the case of $v_2<0$ and $v_3<0$, we conclude that
\begin{multline*}
\beta(t,v_2,v_3)<b-a-(b-a+1)v_2-(b-a+1)bv_3+v_2b+v_3b^2\\=b-a+(a-1)v_2+(a-1)bv_3<0,
\end{multline*} 
thus $\beta(t,v_2,v_3)$ is not totally positive for this choice of $v_2$ and $v_3$.

If $v_2>0$ and $v_3\geq 0$, we see that
\begin{multline*}
(\beta(t,v_2,v_3))''<b-a-(b-a+1)v_2-(b-a+1)bv_3+v_2+v_3\\=b-a-(b-a)v_2-(b(b-a+1)-1)v_3\leq 0,
\end{multline*}
thus we again do not have a totally positive element.

It remains to prove the last part of this lemma, in which we suppose $b\leq 2a-3$. Firstly, we will deal with the case when $v_2<0$ and $v_3=0$. Assuming it, we see that
\[
\beta(t,v_2,0)<b-a-(b-a+1)v_2+v_2b=b-a+(a-1)v_2\leq b-2a+1< 0
\]
for $b\leq 2a-3$. Let now $v_2<0$ and $v_3> 0$. Then 
\begin{multline*}
(\beta(t,v_2,v_3))'<t-(b-a+1)v_2-(b-a+1)bv_3+v_2(a-1)+v_3a^2\\=t-(b-2a+2)v_2-(b(b-a+1)-a^2)v_3;
\end{multline*}
the expression on the right side is positive only if 
\[
v_2>\frac{t-(b(b-a+1)-a^2)v_3}{b-2a+2}.
\]
Likewise, since $2a-2-b$ is positive, we get
\begin{multline*}
(2a-2-b)(\beta(t,v_2,v_3))''<(2a-2-b)(t-(b-a+1)v_2-(b(b-a+1)-1)v_3)
\\\leq (2a-2-b)\bigg(t-(b-a+1)\frac{t-(b(b-a+1)-a^2)v_3}{b-2a+2}-(b(b-a+1)-1)v_3\bigg)\\
=(a-1)t+(-ab^2+b^2+2a^2b-2ab-a^3+a^2+2a-2)v_3.
\end{multline*}
The value of $(a-1)t$ is at most $(a-1)(b-a)$. For $t=b-a$ and $v_3=1$, we get a quadratic polynomial in $b$ with a negative leading coefficient. It has roots $a-1$ and $a+2$. Thus, for $b\geq a+2$, its value is negative or zero. If it is true for $v_3=1$, it must be valid for every $v_3>0$ (the previous considerations clearly imply that the coefficient before $v_3$ is negative). Hence $(\beta(t,v_2,v_3))'$ and $(\beta(t,v_2,v_3))''$ cannot be both positive for this case, which completes the proof of (4).  
\end{proof}

The necessary condition for $\alpha$ being indecomposable is that $\alpha$ is less than $1$ in some embedding. It can be easily shown that $-v+(bw+1)\rho''-w\rho''^2<1$. Thus the elements of the form $\beta(t,0,0)=t$ cannot appear in a totally positive decomposition of $-v+(bw+1)\rho-w\rho^2$, and we can exclude them from the following considerations. Moreover, we will further use this fact to deal with the cases which were omitted in part (4) of the previous lemma.   

\begin{lemma} \label{lem:lessthan1}
Let $b\geq 2a-2$, $v_2<0$ and $v_3\geq 0$. Moreover, let $(a,b)\neq (2,4),(2,5),(3,5)$. If $\beta(t,v_2,v_3)$ is totally positive, then $(\beta(t,v_2,v_3))''>1$. 
\end{lemma}

\begin{proof}
On the contrary, let us suppose $ (\beta(t,v_2,v_3))''\leq 1$. It implies that
\begin{equation} \label{eq:lessthan1}
1\geq (\beta(t,v_2,v_3))''>t-\left(b-a+1-\frac{1}{ab-1}\right)v_2-(b-a+1)bv_3,
\end{equation}
which gives
\[
v_2>\frac{t-1-(b-a+1)bv_3}{b-a+1-\frac{1}{ab-1}}.
\]
Note that if $v_3=0$, then the right side of (\ref{eq:lessthan1}) is clearly greater than $1$ for $v_2<0$, thus we get $(\beta(t,v_2,0))''>1$. Therefore, in the following, we can suppose $v_3>0$.
 
To get a totally positive element, we must have
\begin{align*}
0&<(ab^2-a^2b+ab-b+a-2)(\beta(t,v_2,v_3))'\\
&<(ab^2-a^2b+ab-b+a-2)\bigg(t-(b-2a+2)\frac{t-1-(b-a+1)bv_3}{b-a+1-\frac{1}{ab-1}}\\&\hspace{10cm}-(b(b-a+1)-a^2)v_3\bigg)\\
&=ab^2-2a^2b+2ab-b+2a-2+(a^2b-ab-a)t\\&\qquad+(-a^2b^3+ab^3+2a^3b^2-2a^2b^2+2ab^2-a^4b+a^3b-2a^2b+ab+a^3-2a^2)v_3
\end{align*}
However, similarly as in the proof of part (4) of Lemma \ref{lem:excnegat}, it can be shown that the 
last expression is negative or zero for any choice of $1\leq t\leq b-a$, $v_3\geq 1$ and $(a,b)\neq (2,4),(2,5),(3,5)$. It means that $\beta(t,v_2,v_3)$ is not totally positive if $ (\beta(t,v_2,v_3))''\leq 1$.  
\end{proof}

As we have proved, every totally positive decomposition of integers $-v+(bw+1)\rho-w\rho^2$ can consist only of the elements $\beta(t,v_2,v_3)$ with $v_2>0$ and $v_3<0$. In the following proposition, we will show that even this restricted situation is not possible, which will complete the proof of indecomposability of integers belonging to $\Psi$.

\begin{prop} \label{prop:indeexcunits}
Let $\rho$ be a root of the polynomial
$f(x)=x^3-(a+b)x^2+abx-1$
where $2\leq a\leq b-2$. Then the elements of the form
\[
-v+(bw+1)\rho-w\rho^2 \text{ where } 0\leq v\leq b-a-1 \text{ and } va\leq w\leq(v+1)a-1.
\] 
are, up to multiplication by totally positive units, all the indecomposable integers in $\Z[\rho]$. 
\end{prop}

\begin{proof}
Recall that the set $\Psi$ contains all the candidates on the indecomposable integers in $\Z[\rho]$, and thus it remains to show that all these elements are indecomposable. For $(a,b)=(2,4),(2,5),(3,5)$, we can verify it by a computer program (we used Mathematica). In all the other cases, we use Lemmas \ref{lem:excnegat} and \ref{lem:lessthan1}, which exclude most of the possible decompositions of our elements.  

If our elements were decomposable, then there would exist two totally positive elements $\beta(t,v_2,v_3)$ and $\beta(\bar{t},\bar{v}_2,\bar{v}_3)$ with $v_2,\bar{v}_2>0$ and $v_3,\bar{v}_3<0$ such that 
\[
-v+(bw+1)\rho-w\rho^2=\beta(t,v_2,v_3)+\beta(\bar{t},\bar{v}_2,\bar{v}_3).
\] 
If we consider their coefficients in the basis $1,\rho$ and $\rho^2$, the sum of the coefficients before $1$ must be equal to $-v$, thus at least one of them must be smaller or equal to $-\frac{v}{2}$. Without loss of generality, suppose that it is true for $\beta(t,v_2,v_3)$, i.e., $t-(b-a+1)v_2-(b-a+1)bv_3\leq -\frac{v}{2}$.

Moreover, $\beta(t,v_2,v_3)$ must be totally positive, thus
\[
0<(\beta(t,v_2,v_3))''<t-(b-a+1)v_2-(b-a+1)bv_3+\frac{v_2}{ab-1},
\]
which with the previous condition leads to
\[
- \frac{v_2}{ab-1} < t-(b-a+1)v_2-(b-a+1)bv_3\leq -\frac{v}{2}.
\]
These inequalities can be rewritten as
\begin{equation} \label{eq:forv_3}
\frac{v_2}{b}-\frac{t}{b(b-a+1)}-\frac{v_2}{b(ab-1)(b-a+1)}< -v_3 \leq \frac{v_2}{b}-\frac{v+2t}{2b(b-a+1)}.
\end{equation}
Our next task is to show that except for some exceptional cases of $v_2$, these inequalities cannot be satisfied for any $v_3\in\Z$.

Let us write $v_2$ as $v_2=lb+k$ where $k,l\in \N_0$ and $0\leq k\leq b-1$. Moreover, since $v_2,\bar{v}_2>0$ and $v_2+\bar{v}_2=bw+1$, we must have $0\leq l\leq w$. Using this we get
\[
 l+\frac{k}{b}-\frac{t}{b(b-a+1)}-\frac{v_2}{b(ab-1)(b-a+1)} < -v_3 \leq  l+\frac{k}{b}-\frac{v+2t}{2b(b-a+1)}.
\]
The expression $\frac{k}{b}-\frac{t}{b(b-a+1)}-\frac{v_2}{ab(b-1)(b-a+1)}$ attains its smallest value for $k=0$, $t=b-a$ and $v_2=bw$ (as $v_2< bw+1$), which can be further estimated using $w\leq a(v+1)-1$ and $v\leq b-a-1$. Having this, we get
\[
\frac{k}{b}-\frac{t}{b(b-a+1)}-\frac{v_2}{ab(b-1)(b-a+1)} \geq-\frac{b-a}{b(b-a+1)}-\frac{b(a(b-a)-1)}{b(ab-1)(b-a+1)}>-1
\]
Note that the last inequality is fulfilled since it equivalent to
\[
b^2(ab-a^2-a-1)+2a^2b+ab+b-a>0,
\]
which is satisfied for every $b\geq a+2$ and $a\geq 2$. Hence $-v_3\geq l$. Moreover, likewise, we can deduce that $\frac{k}{b}-\frac{t}{b(b-a+1)}-\frac{v_2}{b(ab-1)(b-a+1)}>0$ for every $k\geq 2$, which gives $-v_3\geq l+1$ in these cases of $k$.  

On the other hand,
\[
\frac{k}{b}-\frac{v+2t}{2b(b-a+1)}\leq \frac{b-1}{b}-\frac{1}{b(b-a+1)}<1,
\]
which implies $-v_3\leq l$. Thus for $k\geq 2$, there is no $v_3\in\Z$ satisfying (\ref{eq:forv_3}). If $k=0$, then $\frac{k}{b}-\frac{v+2t}{2b(b-a+1)}<0$, which leads to the same conclusion. Thus, it remains to solve the case when $k=1$. 

Therefore, we get $-v_3=l$ and $v_2=bl+1$, which implies that
\[
t-(b-a+1)v_2-(b-a+1)bv_3=t-(b-a+1)
\] 
The coefficient $l$ can vary between $0$ and $a(b-a)-1$ as $w$, and thus uniquely corresponds to the integer $-\tilde{v}+(b\tilde{w}+1)\rho-\tilde{w}\rho^2\in\Psi$ with $\tilde{w}=l$. We will show that $-\tilde{v}=t-(b-a+1)$. The value of $t-(b-a+1)$ cannot be larger than $-\tilde{v}$ since in that case, $\beta(t,v_2,v_3)$ would be larger than $1$ in all the embeddings (beacause of the total positivity of $-\tilde{v}+(b\tilde{w}+1)\rho-\tilde{w}\rho^2$), and therefore it cannot appear in a decomposition of our integers from $\Psi$. One the other hand, $t-(b-a+1)$ cannot be smaller that $-\tilde{v}$ since $-\tilde{v}+(b\tilde{w}+1)\rho''-\tilde{w}\rho''^2<1$, and in that case,   $\beta(t,v_2,v_3)$ would not be totally positive. Thus, the only possible candidates on $\beta(t,v_2,v_3)$ are elements from $\Psi$.

Let us fix some $-v+(bw+1)\rho-w\rho^2$; as we have just proved, the only way how this element can decompose is that
\[
-v+(bw+1)\rho-w\rho^2=-\tilde{v}+(b\tilde{w}+1)\rho-\tilde{w}\rho^2+\beta(\bar{t},\bar{v}_2,\bar{v}_3)
\]
for some $0\leq \tilde{v}\leq b-a-1$, $a\tilde{v}\leq \tilde{w}\leq a(\tilde{v}+1)-1$, $\bar{v}_2>0$ and $\bar{v}_3<0$. Recall that $\text{Tr}(\delta_v(-v+(bw+1)\rho-w\rho^2))=b-a+1-v$, which must be strictly greater then $\text{Tr}(\delta_v(-\tilde{v}+(b\tilde{w}+1)\rho-\tilde{w}\rho^2))=b-a+1-\tilde{v}$. This is possible only if $\tilde{v}>v$. However, in that case, we have $b\tilde{w}+1>bw+1$, and thus the sum of $-\tilde{v}+(b\tilde{w}+1)\rho-\tilde{w}\rho^2$ and $\beta(\bar{t},\bar{v}_2,\bar{v}_3)$ cannot be equal to $-v+(bw+1)\rho-w\rho^2$ as $v_2>0$. By this, we have excluded all the remaining possible totally positive decompositions of our elements from $\Psi$ and completed the proof. 
     
\end{proof}

\subsection{Minimal traces} \label{subsec:mintraces}

The goal of this subsection is to determine the precise values of $\min_{\delta\in\Z[\rho]^{\vee,+}}\textup{Tr}(\alpha\delta)$ for $\alpha$ belonging to $\Psi$. Recall that if we multiply $\alpha=-v+(bw+1)\rho-w\rho^2$ by the totally positive element $\delta_v=\frac{1}{f'(\rho)}(a^2-a-(2a-1)\rho+\rho^2)$, we get $\text{Tr}(\alpha\delta_v)=b-a-v+1$. As we will see below, for some cases of $\alpha$, this value is equal to our minimum. However, if we consider the totally positive element
\[
\delta_w=\frac{a^2b-a-1-(ab+a^2-1)\rho+a\rho^2}{f'(\rho)},
\]
we obtain $\text{Tr}(\alpha\delta_w)=w-av+1$. This trace is smaller than $b-a-v+1$ if $w<av+b-a-v$. In particular, if $w=av$, it is equal to $1$. It implies that $\Z[\rho]$ contains at least $b-a$ totally positive elements $\alpha$ satisfying $\min_{\delta\in\Z[\rho]^{\vee,+}}\textup{Tr}(\alpha\delta)=1$, and it is fulfilled for one particular $\delta\in \Z[\rho]^{\vee,+}$.   

To prove that one of these traces is minimal, we will use a different method than in Section \ref{Sec:Ennola}. However, the initial steps will be similar. Let $\delta_t=c+d\rho+e\rho^2\in\Z[\rho]$. It can be easily computed that 
\[
\text{Tr}\left(\frac{\delta_t}{f'(\rho)}(-v+(bw+1)\rho-w\rho^2)\right)=-cw+d(1-aw)+e(a+b-v-a^2w).
\]
This trace is equal to some $t\in\Z$ if and only if $\delta_t$ is of the form 
\[
\delta_t=-at-a(v-b)l+(1-aw)k+(t+(v-b-a)l+wk)\rho+l\rho^2
\]
for some $k,l\in\Z$. In the following proposition, we will show that $\frac{\delta_t}{f'(\rho)}$ is not totally positive for any $k,l\in\Z$ and $t<\min\{b-a-v+1,w-av+1\}$.    

\begin{prop}
Let $\alpha=-v+(bw+1)\rho-w\rho^2\in\Psi$. Then
\[
\min_{\delta\in\Z[\rho]^{\vee,+}}\textup{Tr}(\alpha\delta)=\min\{b-a-v+1,w-av+1\}.
\]
\end{prop}

\begin{proof}

As we have suggested, we will prove that $\frac{\delta_t}{f'(\rho)}$ cannot be totally positive. To reach this goal, we will discuss several different traces, which all have to be positive if it were totally positive. First of all, we can easily compute that 
\[
0<\text{Tr}\bigg(\frac{\delta_t}{f'(\rho)}\bigg)=l,
\] 
which gives us the first condition $l>0$. 

Let us now consider the element of the form $-\bar{v}+(b\bar{w}+1)\rho-\bar{w}\rho^2$. This element is totally positive if it belongs to $\Psi$ (but not just in this case, as we will see below). We can directly compute that
\[
\text{Tr}\left(\frac{\delta_t}{f'(\rho)}(-\bar{v}+(b\bar{w}+1)\rho-\bar{w}\rho^2)\right)=t+l(v-\bar{v})+k(w-\bar{w});
\]
this trace must be positive if $\frac{\delta_t}{f'(\rho)}$ and $-\bar{v}+(b\bar{w}+1)\rho-\bar{w}\rho^2$ are both totally positive. Now we will consider several choices of $\bar{v}$ and $\bar{w}$, which will provide us some specifying information on $k$ and $l$, which, in the end, will lead to a contradiction.

Moreover, we can exclude from our consideration the indecomposable integers with $w=av$ as they attain the minimal trace $1$, which cannot be lowered. Thus suppose $av<w$, and put $\bar{v}=v$ and $\bar{w}=av$. We get
\[
0<\text{Tr}\left(\frac{\delta_t}{f'(\rho)}(-v+(bav+1)\rho-av\rho^2)\right)=t+k(w-av).
\]
Now assume $k<0$. Since $t<\min\{b-a-v+1,w-av+1\}$, and thus $t\leq w-av$, we get
\[
0<t+k(w-av)\leq w-av+k(w-av)\leq 0,
\] 
which is impossible. Hence $k\geq 0$. 

Recall that in our determination of indecomposable integers, we have used the totally positive unit $(\rho-a)^{-2}=a-b+(-ba^2+b^2a+1)\rho+(a^2-ab)\rho^2$. This unit is also of our preferable form since we can set $\bar{v}=b-a$ and $\bar{w}=a(b-a)=a\bar{v}$. Together with $\frac{\delta_t}{f'(\rho)}$, it produces the trace
\[
\text{Tr}\left(\frac{\delta_t}{f'(\rho)}(\rho-a)^{-2}\right)=t+l(v-b+a)+k(w-a(b-a)).
\]
However, this trace is not positive as we have $l>0$, $v-b+a<0$, $k\geq 0$, $w-a(b-a)<0$ and $t\leq b-a-v$. Thus $\frac{\delta_t}{f'(\rho)}$ cannot be totally positive for any $k,l\in\Z$ and \[t<\min\{b-a-v+1,w-av+1\}\]. 
\end{proof}

Moreover, this conclusion provides us an upper bound on this minimal trace for indecomposable integers in $\Z[\rho]$.

\begin{cor}
Let $\alpha$ be an indecomposable integer in $\Z[\rho]$. Then
\[
\min_{\delta\in\Z[\rho]^{\vee,+}}\textup{Tr}(\alpha\delta)\leq\min\{b-a+1,a\}.
\]
\end{cor}

Using these results, we can prove Theorem \ref{thm:mintraces} stated in the introduction. 

\begin{proof}
As we have proved, the maximum of our minimal traces is equal to $\min\{b-a+1,a\}$, and it is attained by some (not necessarily one) indecomposable integer in $\Z[\rho]$. To get our desired result, we have to choose a pair of coefficients $a,b$ such that $a>n$, $b-a+1>n$ and $2\leq a\leq b-2$. Then a root $\rho$ of the polynomial $x^3-(a+b)x^2+abx-1$ satisfies our requierements.

\end{proof}

\subsection{Bound on norm of indecomposable integers} 
In this subsection, we will derive an upper bound on the norm of indecomposable integers in $\Z[\rho]$. In Subsection \ref{subsec:indecom}, we have proved that indecomposable integers in $\Z[\rho]$ are (up to the multiplication by totally positive units) exactly the elements in $\Psi$.
Thus, let us denote \[\alpha(v,W)=-v+(b(av+W)+1)\rho-(av+W)\rho^2\in\Psi\] for some $0\leq W\leq a-1$ (i.e., we put $w=av+W$). We can easily show that
\begin{multline*}
N(\alpha(v,W))=1 - a^2 v + a b v - 2 a v^2 + b v^2 - v^3 - a W + 2 b W - 3 v W - 
 a^2 b v W + a b^2 v W\\ - a b v^2 W - a b W^2 + b^2 W^2 - a v W^2 - 
 b v W^2 - W^3
\end{multline*}

Similarly, as for the simplest cubic fields, we eliminate the smallest norms by comparing them with others.

\begin{lemma}
Let $(a,b)\neq (2,4) $. Moreover, let $a=2A+a_0$ and $b-a=2L+l_0$ where $a_0,l_0\in\{0,1\}$. Then
\begin{enumerate}
\item $N(\alpha(v,W))<N(\alpha(v,W+1))$ for all $0\leq v\leq b-a-2$ and $0\leq W\leq a-2$,
\item if $a\neq 2$, then
\begin{enumerate}
\item $N(\alpha(b-a-1,W))<N(\alpha(b-a-1,W+1))$ for all $0\leq W\leq A$,
\item $N(\alpha(b-a-1,W))<N(\alpha(b-a-2,W))$ for all $A+1\leq W\leq a-1$,
\end{enumerate}
\item if $a=2$, then $N(\alpha(b-3,0))<N(\alpha(b-3,1))$ and $N(\alpha(b-3,1))<N(\alpha(b-4,1))$,
\item $N(v,a-1)<N(v+1,a-1)$
\begin{enumerate}
\item  for all $0\leq v\leq L-2$ and $L\geq 2$ if $l_0=0$,
\item for all $0\leq v\leq L-1$ if $l_0=1$.
\end{enumerate}
\end{enumerate}
\end{lemma} 

\begin{proof}
We will provide only a sketch of the proof of (1); the other parts are analogous, and we can use the same tools as in the proof of Lemma \ref{lem:comsimplest}. We can easily compute that $N(\alpha(v,W+1))-N(\alpha(v,W))$ is a quadratic polynomial in $W$ with a negative leading coefficient, and thus it can be positive only on some convex set. Therefore, it suffices to show that it takes positive values for $W=0$ and $W=a-2$. For $W=0$, we get
\[
N(\alpha(v,1))-N(\alpha(v,0))=-abv^2+(ab^2-a^2b-b-a-3)v+b^2-ab+2b-a-1,
\]
which is again a quadratic polynomial in $v$ with a negative leading coefficient, so we have to discuss whether its values for $v=0$ and $v=b-a-2$ are positive. For $v=0$, we get 
\[
b^2-ab+2b-a-1,
\]
which is obviously positive. For $v=b-a-2$, we obtain
\[
2ab^2-2a^2b-5ab+b+a^2+4a+5.
\]
This is a quadratic polynomial in $b$ with a positive leading coefficient. However, its value for $b=a+2$ is positive as well as its derivative at the same point. Thus this polynomial is positive for every coefficient $b\geq a+2$. We can proceed similarly for $W=a-2$.  

\end{proof}    

As we can see, the previous lemma has excluded most of the candidates on the indecomposable integer with the largest norm. We can thus make the following conclusion. Note that for the particular case of $a=2$ and $b=4$, we can use a computer program to find our maximum.

\begin{prop}
Let $b-a=2L+l_0$ where $l_0\in\{0,1\}$. If $\alpha$ is indecomposable in $\Z[\rho]$, then
\[
N(\alpha)\leq \left\{\begin{array}{ll}
16 & \text{if } a=2 \text{ and } b=4,\\
\max\{N(\alpha(v,a-1)),L-1\leq v\leq 2L-1\} & \text{if } l_0=0, \\
\max\{N(\alpha(v,a-1)),L\leq v\leq 2L\} & \text{if } l_0=1.
\end{array}\right.
\]
\end{prop}

In the following proposition, we will show that if $b-a$ is not too large in comparison to $a$, then the maximal norm is attained for $v$ roughly in the middle of $b-a$.

\begin{prop}
Let $(a,b)\neq (2,4)$ and let $b-a=2L+l_0$ where $l_0\in\{0,1\}$. If $\alpha$ is indecomposable in $\Z[\rho]$, then
\begin{enumerate}
\item $N(\alpha)\leq L^2a^3+(2L^3+L^2+1)a^2-(2L^3+3L^2+L)a+L^3+L^2$ if $l_0=0$ and  $1\leq L\leq a-2$,
\item $N(\alpha)\leq (L^2+L)a^3+(2L^3+4L^2+3L+3)a^2-(2L^3+6L^2+5L+3)a+L^3+3L^2+2L+1$ if $l_0=1$ and $1\leq L\leq a+1$.
\end{enumerate}
\end{prop} 

\begin{proof}
For $b-a=2$, we directly get the maximal norm for $v=0$ and $W=a-1$, which can be expressed by the given formula. 
Thus, in the following, we can assume $b-a\geq 3$.
Since the norm of $\alpha(v,a-1)$ is a cubic polynomial in $v$ with a negative leading coefficient, it suffices to compare the norms $N(\alpha(L-1,a-1))$ and $N(\alpha(L,a-1))$ for $l_0=0$, and $N(\alpha(L,a-1))$ and $N(\alpha(L+1,a-1))$ for $l_0=1$. For the considered cases of $L$, the first norms are always larger.  
\end{proof}

\section*{Acknowledgements}
The author wishes to express her thanks to Vítězslav Kala for many helpful comments.


\begin{thebibliography}{10} 

\bibitem{Ba} S. Balady, \textit{Families of cyclic cubic fields}, J. Number Theory 167, 394--406 (2016).

\bibitem{BK} V. Blomer and V. Kala, \textit{Number fields without $n$-ary universal quadratic forms}, Math. Proc. Cambridge Philos. Soc. 159(2), 239--252 (2015). 

\bibitem{Bru} H. Brunotte, \textit{Zur Zerlegung totalpositiver Zahlen in Ordnungen totalreeller algebraischer Zahlk\"orper}, Arch. Math. (Basel) 41(6), 502--503 (1983).

\bibitem{By} D. Byeon, \textit{Class number 3 problem for the simplest cubic fields}, Proc. Amer. Math. Soc. 128, 1319--1323 (2000). 

\bibitem{CLSTZ} M. \v{C}ech, D. Lachman, J. Svoboda, M. Tinkov\'a and K. Zemkov\'a, \textit{Universal quadratic forms and indecomposables over biquadratic fields}, Math. Nachr. 292, 540--555 (2019). 

\bibitem{Cohen} H. Cohen, \textit{Variations sur un thème de Siegel et Hecke}, Acta Arith. 30, 63--93 (1976).  

\bibitem{Co} H. Cohn, \textit{A device for generating fields of even class number}, Proc. Amer. Math. Soc. 7, 595--598 (1956).

\bibitem{DS} A. Dress and R. Scharlau, \textit{Indecomposable totally positive numbers in real quadratic orders}, J. Number Theory 14, 292--306 (1982).

\bibitem{En2} V. Ennola, \textit{Fundamental units in a family of cubic fields}, J. Th\'eor. Nombres Bordeaux 16, 569--575 (2004).

\bibitem{Fo}  K. Foster, \textit{HT90 and ``simplest'' number fields},  Illinois J. Math. 55, 1621--1655 (2011).

\bibitem{HK} T. Hejda and V. Kala, \textit{Additive structure of totally positive quadratic integers}, Manuscripta Math. 163, 263--278 (2020).

\bibitem{JK} S. W. Jang and B. M. Kim, \textit{A refinement of the Dress-Scharlau theorem}, J. Number Theory 158, 234--243 (2016).

\bibitem{Ka} V. Kala, \textit{Universal quadratic forms and elements of small norm in real quadratic fields}, Bull. Aust. Math. Soc. 94, 7--14 (2016).

\bibitem{Ka2} V. Kala, \textit{Norms of indecomposable integers in real quadratic fields}, J. Number Theory 166, 193--207 (2016).

\bibitem{KT} V. Kala and M. Tinková, \textit{Universal quadratic forms, small norms and traces in families of number fields}, preprint. https://arxiv.org/abs/2005.12312

\bibitem{KY} V. Kala and P. Yatsyna, \textit{Lifting problem for universal quadratic forms}, Adv. Math. 377, 107497 (2021).

\bibitem{KTZ} J. Krásenský, M. Tinková and K. Zemková, \textit{There are no universal ternary quadratic forms over biquadratic fields}, Proc. Edinb. Math. Soc. 63 (3), 861-912 (2020).

\bibitem{KH} H. K. Kim and H. J. Hwang, \textit{Values of zeta functions and class number 1 criterion for the simplest cubic fields}, Nagoya Math. J. 160, 161--180 (2000).

\bibitem{Kis}  Y. Kishi, \textit{A family of cyclic cubic polynomials whose roots are systems of fundamental units}, J. Number Theory 102(1), 90--106 (2003).

\bibitem{LP} F. Lemmermeyer and A. Peth\"{o}, \textit{Simplest Cubic Fields}, Manuscripta Math. 88, 53--58 (1995).

\bibitem{Let} G. Lettl, \textit{A lower bound for the class number of certain cubic number fields}, Math. Comp. 46, 659--666 (1986).

\bibitem{Lo} S. Louboutin, \textit{Class-number problems for cubic number fields}, Nagoya Math. J. 138, 199--208 (1995).

\bibitem{Na} W. Narkiewicz, \textit{Elementary and analytic theory of algebraic numbers}, 3rd Edition, Springer-Verlag, Berlin, 2004.

\bibitem{Ok} R.  Okazaki, \textit{An elementary proof for a theorem of Thomas and Vasquez}, J. Number Theory 55, 197--208 (1995). 

\bibitem{Pe} O. Perron, \textit{Die Lehre von den Kettenbr\"uchen}, B. G. Teubner, 1913.

\bibitem{Sh} D. Shanks, \textit{The simplest cubic number fields}, Math. Comp. 28, 1137--1152 (1974).

\bibitem{Si} C. L. Siegel, Sums of m-th powers of algebraic integers, Ann. of Math. 46, 313--339 (1945).  

\bibitem{Th} E. Thomas, \textit{Fundamental units for orders in certain cubic number fields}, J. Reine Angew. Math. 310, 33--55 (1979).

\bibitem{ThV} E. Thomas and A. T. Vasquez, \textit{On the resolution of cusp singularities and the Shintani decomposition in totally real cubic number fields}, Math. Ann. 247, 1--20 (1980).

\bibitem{TV} M. Tinkov\'a and P. Voutier, \textit{Indecomposable integers in real quadratic fields}, J. Number Theory 212, 458--482 (2020).

\bibitem{Wa} L. Washington, \textit{Class numbers of the simplest cubic fields},
Math. Comp. 48, 371--384 (1987).


\bibitem{Ya} P. Yatsyna, \textit{A lower bound for the rank of a universal quadratic form with integer coeficients in a totally real field}, Comment. Math. Helvet. 94, 221--239 (2019).

\end{thebibliography}
\end{document}